\documentclass[12pt,a4paper]{amsart}

\usepackage{graphicx}
\usepackage{color}
\usepackage{url}
\usepackage{amssymb}
\usepackage{cite}

\usepackage{graphicx}
\usepackage{color}
\usepackage{hyperref}

\newtheorem{theorem}[equation]{Theorem}
\newtheorem{lemma}[equation]{Lemma}
\newtheorem{prop}[equation]{Proposition}
\newtheorem{cor}[equation]{Corollary}

\theoremstyle{definition}

\newtheorem{example}[equation]{Example}

\theoremstyle{remark}
\newtheorem{remark}[equation]{Remark}

\numberwithin{equation}{section}

\newcommand{\rr}{{\mathbb R}}

\newcommand{\rd}{{\mathbb R^m}}

\newcommand{\supp}{\operatorname{supp}}

\newcommand{\ip}[2]{{\langle #1,#2\rangle}}

\newcommand{\Sd}{S^{m-1}}
\newcommand{\R}{\mathbb R}
\newcommand{\Rd}{\mathbb R^m}
\newcommand{\disc}{\text{disc}}
\newcommand{\eps}{\varepsilon}
\newcommand{\spctim}{\mathbb R^{m+1}}
\newcommand{\pr}{\mathbf P}
\newcommand{\E}{\mathbf E}

\begin{document}

\title{Limit theorems and governing equations for L\'evy walks}

\author{M. Magdziarz}

\author{H.P. Scheffler}

\author{P. Straka}

\author{P. Zebrowski}

\begin{abstract}
 The L\'evy Walk is the process with continuous sample paths which arises from consecutive linear motions of i.i.d.\ lengths with i.i.d.\  directions.
 Assuming speed $1$ and motions in the domain of $\beta$-stable attraction, we prove functional limit theorems and derive governing pseudo-differential equations for the law of the walker's position.
 Both L\'evy Walk and its limit process are continuous and ballistic in the case $\beta \in (0,1)$.
 In the case $\beta \in (1,2)$, the scaling limit of the process is $\beta$-stable and hence discontinuous. This case exhibits an interesting situation in which scaling exponent $1/\beta$ on the process level is seemingly unrelated to the scaling exponent $3-\beta$ of the second moment.
 For $\beta = 2$, the scaling limit is Brownian motion.
\end{abstract}

\maketitle

\section{Introduction}
The Continuous Time Random Walk (CTRW) is a stochastic process determined uniquely by $\mathbb{R}^m$-valued i.i.d.\ random vectors
${Y_1}$, ${Y_2},\dots$ representing consecutive jumps of the random walker,
and $\rr_+$-valued i.i.d. random variables $J_1,J_2,\dots$ representing waiting times between jumps.
The trajectories of CTRWs are step functions with intervals $J_i$ and jumps $Y_i$. Taking $J_1=J_2= {\dots}  = 1$ we obtain the classical random walk process.
CTRWs were introduced for the first time in the pioneering work of Montroll and Weiss \cite{Weiss}.
Since then they became one of the most popular and useful models in statistical physics \cite{Sokolov_book}.
Their first spectacular application can be found in \cite{Montroll}, in which
CTRWs with heavy-tailed waiting times were used as a model of charge carrier transport in amorphous semiconductors.
Today CTRWs are well established mathematical models, particularly
attractive in the modeling of anomalous dynamics characterized by
nonlinear in time mean square displacement $\text{Var}(X(t))\sim t^\alpha$, $\alpha\neq 1$, see \cite{Metzler_Klafter} and references therein.

The main mathematical challenges in the study of
CTRWs are limit theorems on the stochastic process level and governing equations
describing evolution in time of the corresponding probability density functions. There is an extensive literature
in this field: general results for the scaling limits of CTRWs on the stochastic process level can be found in \cite{CTRW1,coupleCTRW,gammagt1,Kolokoltsov09,StrakaHenry,Straka,Sikorskii}.
Governing equations for the densities of the CTRW limits and the related fractional Cauchy problems were analyzed in \cite{Baeumer2001,Baeumer,Sikorskii,Straka,Jurlewicz,Nane,Hahn11}.
Some recent results for particular classes of correlated and coupled CTRWs
can be found in \cite{Jurlewicz2,MM1,MM2,MM3,Barkai2}

The trajectories of CTRW are step functions, thus they are discontinuous.
However, the usual physical requirement for a mathematical model
is to have continuous realizations. 
The straightforward remedy
for this problem is to interpolate the ``stair,'' which results from a waiting time and an instantaneous jump, by a linear motion during the waiting interval.
As a result one obtains a stochastic process with continuous, piecewise linear trajectories -- a proper model of a physical system. 

Although much is known about asymptotic properties of CTRWs,
there are no such results for interpolated CTRWs.
The only one exception that we are aware of are the classical random walks.
In this case the linearly interpolated version
has exactly the same limit (in the $M_1$ topology) as the random walk itself, see Corollary 6.2.1 in \cite{Whitt2010}.
As our results will show, the situation can be drastically different for CTRWs and
their linearly interpolated counterparts -- the limits can differ significantly. 

In this paper we will concentrate on the class of \emph{L\'evy walks}.
Let us take CTRW satisfying $\|Y_i\| = J_i$ for every $i\in\mathbb{N}$ (length of the jump equal to the length of the preceding waiting time).
Next, take the linearly interpolated version of this CTRW.
The obtained process is called a L\'evy walk (see the next section for its formal definition).
It is important to note that in the literature the term L\'evy Walk is also used for \emph{uninterpolated} CTRWs with 
$\|Y_i\| = J_i$; in fact the first definition of L\'evy Walks in \cite{Klafter1} was of this kind (see also \cite{Klafter2}) in the framework
of generalized master equations. Since then, the L\'evy walk became
one of the most popular models in statistical physics with large number of important applications.
The main idea underlying the L\'evy walk
is the coupled spatial-temporal memory in the CTRW, which is manifested by the condition $\|Y_i\| = J_i$.
Thus, even if we assume that the distribution of jumps $Y_i$ is heavy-tailed with diverging moments,
the L\'evy walk itself has finite moments of all orders (intuitively, long jumps
are penalized by requiring more time to be performed). 
This is very different from
$\alpha$-stable L\'evy processes with $\alpha<2$, which have infinite variance.
Moreover, the trajectories of the L\'evy walk are continuous
and piecewise linear. 
These desirable properties make the L\'evy walk particularly attractive for physical applications.
L\'evy walks have been found to be excellent models in the description of various real-life phenomena and complex anomalous systems. The most striking examples are: transport of light in optical materials \cite{Wiersma},
foraging patterns of animals \cite{Bell,Berg,Buchanan}, epidemic spreading \cite{Brockmann2,Dybiec},
human travel \cite{Brockmann,Gonzales}, blinking nano-crystals \cite{Margolin}, and
fluid flow in a rotating annulus \cite{Solomon}.

In this paper we prove functional limit theorems and derive governing equations for L\'evy walks. 
To the best of our knowledge, this is the first systematic
study of linearly interpolated CTRWs and their limits.
Our results show that, in contrast to the standard random walk,
the scaling limit of a linearly interpolated CTRW can be different from the scaling limit of the corresponding CTRW. 
Our results are presented in the following way: 
We introduce the model in Section 2 and clarify some distributional scaling limits (stable laws) in Section 3. 
In Section 4 we prove the weak convergence of rescaled L\'evy Walks  in Skorokhod space. 
Finally, Section 5 contains a derivation of the pseudo-differential governing equation for the probability density of the  scaling limit process.

\section{The model}
Let $\Sd=\{x\in\rd : \|x\|_2=1\}$ be the unit-sphere and let $\lambda$ be any distribution on $\Sd$ modeling the {\it directions} of the random walker. To avoid degenerate cases we further assume that $\supp(\lambda)$ spans $\rd$, where $\supp(\lambda)$ denotes the support of $\lambda$. Let $\Lambda_1,\Lambda_2,\dots$ be i.i.d. on $\Sd$ with distribution $\lambda$. Moreover, let $J_1,J_2,\dots$ be i.i.d. $\rr_+$-valued random variables modeling the random moving times. We also assume that $(\Lambda_i)$ and $(J_i)$ are independent.
Set $T(0)=0$ and $T(n)=\sum_{i=1}^nJ_i$. Then $T(0), T(1), T(2), \dots$ are the times where the walker changes direction. 
Set $S(0)=0$ and let $S(n)=\sum_{j=1}^nJ_j\Lambda_j$, the position after the $n$-the movement. Let $N_t=\max\{n\geq 0 : T(n)\leq t\}$.
Then the {\it L\'evy Walk} $W = \{W(t)\}_{t\geq 0}$ is defined as the stochastic process
\begin{equation}\label{eqm2}
W(t)=\sum_{j=1}^{N_t}J_j\Lambda_j+\bigl(t-T(N_t)\bigr)\Lambda_{N_t+1} .
\end{equation}
Observe that $W(t)$ is right-continuous by construction. Moreover we have $N_t=n$ if and only if $T(n)\leq t<T(n+1)$.
Hence, as $t\uparrow T(n+1)$ we have
\begin{multline*}
W(t)=\sum_{j=1}^nJ_j\Lambda_j+(t-T(n))\Lambda_{n+1}\\  \to \sum_{j=1}^nJ_j\Lambda_j+(T(n+1)-T(n))\Lambda_n = \sum_{j=1}^{n+1}J_j\Lambda_j=W(T(n+1)) ,
\end{multline*}
which shows that $W$ has \emph{continuous} sample paths. Moreover, by construction we have $\|W(t)\|\leq t$ and hence the moments of $W(t)$ are finite for all orders, even if $E(J_1)=\infty$.

\section{Distributional assumptions}

We assume that the distribution of $J_i$ belongs to the domain of attraction \cite{thebook} of some $\beta$-stable law with $\beta \in (0,1)$. 
This means that there exist $b_n>0$ such that we have the following weak convergence of random variables:
\begin{equation}\label{eqm1}
b_nT(n)\Longrightarrow D\quad\text{as $n\to\infty$,}
\end{equation}
where $D$ follows a $\beta$-stable law supported on $[0,\infty)$ with Laplace-transform $E[e^{-sD}]=e^{-s^\beta}$. Here $\Longrightarrow$ denotes convergence in distribution.
The law of $D$ is infinitely divisible with L\'evy measure
\begin{equation}\label{eq0}
\phi_D(dt)=\frac{\beta}{\Gamma(1-\beta)}t^{-\beta-1}dt,
\end{equation}
and the remaining parts of the L\'evy triplet vanish.
The following result provides the joint convergence of $(S(n),T(n))$:

\begin{theorem}\label{thm1}
We have the following weak convergence of probability measures on $\rd\times\rr_+$:
\begin{equation}\label{eq1}
\bigl(b_nS(n),b_nT(n)\bigr)\Longrightarrow (A,D)\quad\text{as $n\to\infty$}.
\end{equation}
The limiting random variable $(A,D)$ has a full (meaning ``not supported on a proper hyperplane'') multivariate $\beta$-stable law on $\rr^{m+1}$. 
Its 
L\'evy-measure $\phi$ is given by
\begin{equation}\label{eq2}
\phi(dr,d\theta,dt)=\varepsilon_t(dr)\lambda(d\theta)\phi_D(dt),
\end{equation}
where the $x$-coordinate of $(x,t) \in \mathbb R^{m+1}$ is given in polar form $x=r\cdot\theta$, $r > 0$, $\theta \in S^{m-1}$ and $\eps_t$ denotes the Dirac measure. 
Moreover, $A$ is multivariate $\beta$-stable on $\rd$ with spectral measure $\lambda$ and L\'evy-measure
\begin{equation}\label{eq3}
\phi_A(dr,d\theta)=\frac{\beta}{\Gamma(1-\beta)}r^{-\beta-1}dr\,\lambda(d\theta) .
\end{equation}
\end{theorem}

\begin{proof}
For $u>0$ and a Borel-set $V\subset\Sd$ let
\[ B(u,V)=\bigl\{x\in\rd : \|x\|_2>u , \frac x{\|x\|_2}\in V\bigr\} \]
and write $\mathcal \partial V$ for the topological boundary of $V$.
Then according to \cite[Th.~3.2.2]{thebook} \eqref{eq1} holds if we can show  for any $u,s>0$ and  $V$ with $\lambda(\partial V)=0$ that
\begin{equation}\label{eq4}
nP\bigl\{b_n J_1 \Lambda_1\in B(u,V), b_nJ_1>s\bigr\} \to \phi\bigl(B(u,V)\times (s,\infty)\bigr)\quad\text{as $n\to\infty$}
\end{equation}
where $\phi$ is given by \eqref{eq2}. Now \eqref{eqm1} and \eqref{eq0} imply that
\[ nP\{b_nJ_1>s\}\to\phi_D(s,\infty)=\frac{s^{-\beta}}{\Gamma(1-\beta)}\quad\text{as $n\to\infty$.} \]
Hence, by independence of $J_1$ and $\Lambda_1$ we obtain
\begin{equation*}
\begin{split}
nP\bigl\{b_nJ_1 \Lambda_1\in B(u,V), b_nJ_1>s\bigr\}  &=nP\bigl\{b_nJ_1>u, \Lambda_1\in V, b_nJ_1>s\bigr\} \\
&= nP\bigl\{b_nJ_1>\max(u,s)\bigr\}\cdot P\bigl\{\Lambda_1\in V\bigr\} \\
&\to \frac{\max(u,s)^{-\beta}}{\Gamma(1-\beta)}\lambda(V) .
\end{split}
\end{equation*}
On the other hand we get from \eqref{eq2} that
\begin{equation*}
\begin{split}
\phi\bigl(B(u,V)\times(s,\infty)\bigr) & = \int_0^\infty\int_{\Sd}\int_0^\infty 1_{B(u,V)}(r\cdot\theta)1_{(s,\infty)}(t)\,\varepsilon_t(dr)\,\lambda(d\theta)\phi_D(dt) \\
&= \int_0^\infty\int_{\Sd}1_{B(u,V)}(t\cdot\theta)1_{(s,\infty)}(t)\,\lambda(d\theta)\phi_D(dt) \\
&=\lambda(V)\int_0^\infty 1_{(\max(u,s),\infty)}(t)\,\phi_D(dt) \\
&=\frac{\max(u,s)^{-\beta}}{\Gamma(1-\beta)}\lambda(V) ,
\end{split}
\end{equation*}
so \eqref{eq4} holds. For the proof of \eqref{eq3} observe that
\begin{equation*}
\begin{split}
\phi_A\bigl(B(u,V)\bigr) &= \phi\bigl(B(u,V)\times\rr_+\bigr) \\
&= \int_0^\infty\int_{\Sd}\int_0^\infty 1_{B(u,V)}(r\cdot\theta)\,\varepsilon_t(dr)\,\lambda(d\theta)\,\phi_D(dt) \\
&=\int_0^\infty\int_{\Sd}1_{B(u,V)}(t\cdot\theta)\,\lambda(d\theta)\,\phi_D(dt) \\
&= \lambda(V)\phi_D(t,\infty).
\end{split}
\end{equation*}
The assumption on the support of $\lambda$ implies that $\supp(\phi)$ spans $\rr^{m+1}$ and hence $(A,D)$ has a full distribution. 
This concludes the proof.
\end{proof}

Since the limit $(A,D)$ in \eqref{eq1} is infinitely divisible, we know from Lemma 2.1 in \cite{coupleCTRW} that the Fourier-Laplace transform of $(A,D)$ for $k\in\rd$ and $s\geq 0$ is given by
\[ \bar{P}_{(A,D)}(k,s):=\int_\rd\int_0^\infty e^{-st}e^{i\ip kx}\,P_{(A,D)}(dx,dt) = e^{-\psi(k,s)} ,\]
where
\[ \psi(k,s)=\int_\rd\int_0^\infty\bigl(1-e^{-st}e^{i\ip kx}\bigr)\,\phi(dx,dt)\]
and $\phi(dx,dt)$ is the L\'evy measure of $(A,D)$.

\begin{lemma}\label{lem1}
We have for $k\in\rd$ and $s\geq 0$ that
\begin{equation}\label{eq5}
\psi(k,s)=\int_{\Sd}\bigl(s-i\ip k\theta\bigr)^\beta\,\lambda(d\theta) .
\end{equation}
Especially we have
\begin{equation*}
\begin{split}
\psi_A(k) & = \psi(k,0)=\int_{\Sd}\bigl(-i\ip k\theta\bigr)^\beta\,\lambda(d\theta) \\
\psi_D(s)&=\psi(0,s)=s^\beta .
\end{split}
\end{equation*}
\end{lemma}

\begin{proof}
Using \eqref{eq2} we have
\begin{equation*}
\begin{split}
\psi(k,s)&=\int_0^\infty\int_{\Sd}\int_0^\infty\bigl(1-e^{-st}e^{i\ip k{r\theta}}\bigr)\,\phi(dr,d\theta,dt) \\
&=\int_0^\infty\int_{\Sd}\bigl(1-e^{-st}e^{i\ip k{t\theta}}\bigr)\,\lambda(d\theta)\,\phi_D(dt) \\
&=\int_{\Sd}\int_0^\infty\bigl(1-e^{-t(s-i\ip k\theta)}\bigr)\,\phi_D(dt)\,\lambda(d\theta) \\
&=\int_{\Sd}\bigl(s-i\ip k\theta\bigr)^\beta\,\lambda(d\theta) .
\end{split}
\end{equation*}
\end{proof}

Note that the symbol $\psi(k,s)$ in \eqref{eq5} defines a pseudo-differential operator, which by the form of the L\'evy measure in \eqref{eq2}, for bounded $C^1$ functions takes the form
\begin{equation}\label{eqps1}
\psi(i\nabla_x,\partial_t)f(x,t) = \int_{\Sd}\bigl(\partial_t+\ip\theta{\nabla_x}\bigr)^\beta f(x,t)\,\lambda(d\theta)
\end{equation} 
where $(\partial_t+\ip\theta{\nabla_x})^\beta$ is the {\it fractional material derivative} in direction $\theta$, given by
\begin{equation}\label{eqps2}
\bigl(\partial_t+\ip\theta{\nabla_x}\bigr)^\beta f(x,t)=\frac 1{\Gamma(1-\beta)}\int_0^\infty \bigl(f(x+r\theta,t+r)-f(x,t)\bigr)r^{-\beta-1}\,dr .
\end{equation}
See the references \cite{M3HPStri,BMMST,Sokolov_Metzler} for details. 
\medskip

If we let $b(c)=b_{[c]}$ then $b(c)$ is regularly varying with index $-1/\beta$. Moreover it follows from Theorem 16.14 in \cite{kallenberg} that \eqref{eq1} implies convergence of the corresponding stochastic processes
\begin{equation}\label{eqP1}
\bigl\{(b(c)S(\lfloor ct\rfloor),b(c)T(\lfloor ct\rfloor))\bigr\}_{t\geq 0}\overset{J_1}{\longrightarrow} \bigr\{(A(t),D(t))\bigr\}_{t\geq 0}
\end{equation}
as $c\to\infty$, where $\{(A(t),D(t))\bigr\}_{t\geq 0}$ is the L\'evy process associates to $(A,D)$ with FLT $\bar{P}_{(A(t),D(t))}(k,s)=e^{-t\psi(k,s)}$ where $\psi(k,s)$ is given by \eqref{eq5}.

\begin{prop}\label{psupp}
\begin{enumerate}
\item
For any $u>0$ the distribution of $(A(u),D(u))$ is supported on the cone 
\[ K=\bigl\{(x,t)\in\rd\times\rr_+ : \|x\|\leq t, t\geq 0\bigr\}. \]
\item
The probability law of $(A(1),D(1)) = (A,D)$ admits a Lebesgue density
$p(x,t)$.
\item
$(A(u),D(u))$ is Lebesgue absolutely continuous with density
$$
p_u(x,t) = u^{-(m+1)/\beta} p(u^{-1/\beta}x,u^{-1/\beta}t).
$$ 
\item
For all $(x,t)$ in the interior of $K$, the density $p_u(x,t)>0$ is strictly positive. 
\end{enumerate}
\end{prop}

\begin{proof}
Observe that by \eqref{eq2} the L\'evy measure of $(A,D)$ is supported on 
\[ S=\bigl\{(x,t)\in\rd\times\rr_+ : \|x\|=t, t\geq 0\bigr \}. \]
Moreover, by assumption on $\lambda$ we know that $\supp(\phi)$ spans $\rd\times \mathbb R_+$. Since the closed convex cone generated by $S$ equals $K$, assertions (1), (2) and (4) follow from Theorem 2 in \cite{pms1}.
(3) is a simple consequence of the stability property.
\end{proof}

\section{The Limit Theorem}

In this section, we give a functional limit theorem for L\'evy walks, i.e.\ we prove scaling limit theorems for $\{W(t)\}_{t \ge 0}$ on the level of stochastic processes,
using the continuous mapping theorem on Skorokhod spaces of sample paths.

We first clarify our notation.
For a separable metric space $E$ (here, $E$ will be $\mathbb R^d$ or $\R$),
write $\mathbb D(E)$ for the set of r.c.l.l.\ paths in $E$.
By this we mean the maps from $[0,\infty)$ to $E$ which are right-continuous with left-hand limits.
For $x \in \mathbb D(E)$, write $x^-:t \mapsto x(t-)$ if $t > 0$, and set $x^-(0) = x(0)$; we call $x^-$ the l.c.r.l.\ (or ``left-continuous with right-hand limits'') version of $x$.
Write $\mathbb D_{u,\uparrow} \subset \mathbb D(\mathbb R)$ for all unbounded and non-decreasing paths $y$ in $\mathbb R$.
For $y \in \mathbb D_{u,\uparrow}$, define the \textit{inverse path}
$y^{-1}(t) = \inf\{s : y(s) > t\}$.
We note that $y^{-1}$ is r.c.l.l.\ and $y^{-1} \in \mathbb D_{u,\uparrow}$.
If $x\in \mathbb D(\mathbb R^d)$ and $y \in \mathbb D_{u,\uparrow}$, then it can be checked that the composition $x\circ y$ is again r.c.l.l.\ and thus defines an element in $ \mathbb D(\mathbb R^d)$.
Moreover, $x^-\circ y^-$ is l.c.r.l. For its r.c.l.l.\ version we write
$(x^-\circ y^-)^+
$, and note that it is defined via $(x^- \circ y^-)^+(t) = (x^- \circ y^-)(t+)$.
Finally, we write $\disc(x) := \{t \ge 0: x^-(t) \neq x(t)\}$ for the set of discontinuities of a r.c.l.l.\ or l.c.r.l.\ path $x$.

Throughout this section, we let
\begin{align*}
 a &\in \mathbb D(\Rd)
& d &\in \mathbb D_{u,\uparrow}
& e &= d^{-1}
\\
x &= (a^- \circ e^-)^+
& y &= a \circ e
\\
g &= (d^- \circ e^-)^+
& h&= d \circ e
\end{align*}
and write $\mathcal R(d):= \{ d(t): t \ge 0 \} $ for the range of $d$.
We say that $t$ is a left-limit point of $\mathcal R(d)$ if there exists a sequence $\{t_n\} \subset \mathcal R(d)$ such that $t_n < t$ and $\lim_{n \to \infty} t_n = t$.
``Right-limit point''
is then defined similarly.
Recall that $e \in \mathbb D_{u, \uparrow}$.  We say that $e$ is right-increasing at $t$ if $u > t \Rightarrow e(u) > e(t)$, and left-increasing if $0 \le s < t \Rightarrow e(s) < e(t)$.
For later use, we note that it can be checked that
\begin{align*}
 t \text{ is a right-limit point of } \mathcal R(d)
 \Leftrightarrow e \text{ is right-increasing at } t
 \Leftrightarrow h(t) = t,
 \\
 t \text{ is a left-limit point of } \mathcal R(d)
 \Leftrightarrow e^- \text{ is left-increasing at } t
 \Leftrightarrow g^-(t) = t
\end{align*}

We define the L\'evy Walk path mapping as follows:
\begin{align*}
 \Phi : \mathbb D(\Rd) \times \mathbb D_{u, \uparrow} &\to \mathbb D(\Rd)
 \\
 (a,d) &\mapsto w
\end{align*}
where
\begin{align} \label{eq:levywalkpath}
w(t)=
\begin{cases}
x(t) & \text{ if } t \in \mathcal R(d) \\
x(t) + \dfrac{t-g(t)}{h(t)-g(t)} (y(t) - x(t)) & \text{ if } t\notin \mathcal R(d)
\end{cases}
\end{align}
Lemma \ref{lem:well-defined} below shows that $\Phi$ is well-defined.
We will see later that $\Phi$ maps the paths of the rescaled cumulative process
$\left( b_n S(\cdot), b_n T(\cdot) \right)$ to the paths of the L\'evy Walk $W$.
We also note that since $g(t) \le t \le h(t)$, the fraction in \eqref{eq:levywalkpath} lies in the interval $[0,1]$ for every $t \ge 0$.  Hence we have
\begin{align} \label{eq:interpol}
 w(t) \in [x(t), y(t)], \quad t \ge 0,
\end{align}
where $[x(t),y(t)]$ denotes the compact linear segment $\subset \Rd$ of points between $x(t)$ and $y(t)$.
The experienced reader may think of $x$ as a CTRW path, and of $y$ as an overshooting CTRW path (or lagging and leading CTRW path).

\begin{lemma} \label{lem:well-defined}
 If $t \notin \mathcal R(d)$, then $h(t) > g(t)$. Moreover, $w$ is r.c.l.l.
\end{lemma}
\begin{proof}
Since $d$ is non-decreasing and r.c.l.l., $\mathcal R(d)$ is right-closed \cite[p.146]{Hof69}, i.e.\ closed in the topology generated by all intervals of the form $[u,v) \subset [0,\infty)$.
It follows that $t \notin \mathcal R(d)$ cannot be a right-limit point of $\mathcal R(d)$.
This implies $t < h(t)$, and clearly $g(t) \le t$.

For the remaining statement, let $ t\ge 0 $.
Throughout, $t_n$ is a sequence with $\lim t_n = t$. We write $t_n \uparrow t$ if additionally $t_n < t$, and $t_n \downarrow t$ if $t_n > t$.

We consider the following exhaustive list of cases:

Case 1: $t \notin \mathcal R(d)$, $t$ is neither left nor right limit point of $\mathcal R(d)$.
It follows that $e$ and hence $g, h, x$ and $y$ are constant in a neighbourhood of $t$.
Thus $w$ is continuous at $t$.

Case 2: $t \notin \mathcal R(d)$, $t$ a left-limit point of $\mathcal R(d)$ (but not a right-limit point). Let $t_n \downarrow t$.
Then for large $n$, $t_n \notin \mathcal R(d)$, and $g(t_n) = t$, $h(t_n) = h(t) > t$, and thus $w(t_n) \to x(t) = w(t)$ since $x$ is r.c.l.l., and it follows that $w$ is right-continuous at $t$.
Now let $t_n \uparrow t$.
We observe that $e(t_n) \uparrow e^-(t)$, and hence $y(t_n) = a(e(t_n)) \to a^-(e^-(t)) = x^-(t)$, which means $x^-(t) = y^-(t)$.
But then \eqref{eq:interpol} implies that the left-hand limit of $w$ at $t$ exists and equals $w^-(t) = x^-(t)$.

Case 3:
$t \in \mathcal R(d)$, $t$ is neither left nor right limit point of $\mathcal R(d)$ (i.e.\ an isolated point).  Let $t_n \downarrow t$.  For large $n$, $t_n \notin \mathcal R(d)$.
Then $h(t_n) = h(t) > t$ and $g(t_n) = g(t)$. It follows that $w(t_n) \to x(t) = w(t)$, and $w$ is right-continuous at $t$.
Now let $t_n \uparrow t$. Again for large $n$, $t_n \notin \mathcal R(d)$.
Then $g(t_n) = g^-(t) < t$ and $h(t_n) = t$.  The fraction converges to $1$,
and $\lim w(t_n) = y^-(t)$ exists.

Case 4:
$t \in \mathcal R(d)$, $t$ is a right-limit point but not a left-limit point of $\mathcal R(d)$.
Let $t_n \downarrow t$.
We find that $e(t_n) \downarrow e(t)$, which implies $e^-(t_n) \downarrow e(t)$.
But then $a^-(e^-(t_n)) \to a(e(t)) = y(t)$, and thus $x(t) = y(t)$.
Then \eqref{eq:interpol} implies that $\lim w(t_n)$ exists and equals $x(t) = w(t)$, and $w$ is right-continuous at $t$.
Now let $t_n \uparrow t$. Proceeding as in Case 3, we see that $\lim w(t_n)$ exists and equals $y^-(t)$.

Case 5:
$t\in \mathcal R(d)$ is a left-limit point but not a right-limit point of $\mathcal R(d)$.
For $t_n \downarrow t$, proceed as in Case 3, and for $t_n \uparrow t$ proceed as in Case 2.

Case 6:
$t \in \mathcal R(d)$ is both left- and right-limit point of $\mathcal R(d)$.
For $t_n \downarrow t$, proceed as in Case 4, and for $t_n \uparrow t$, proceed as in Case 2.
\end{proof}

For later use, we give the following sufficient condition for the continuity of L\'evy Walk sample paths:

\begin{lemma} \label{lem:wcont}
Let $d$ be strictly increasing and assume $\disc(a) \subset \disc(d)$.
Then the path $w$ is continuous.
\end{lemma}
\begin{proof}
We reiterate Cases 1-6 from the proof of Lemma \ref{lem:well-defined}, and show that $w^-(t) = w(t)$ in each case.  Since $d$ is strictly increasing, $e$ is now continuous, i.e.\ $e^- = e$.

Case 1: nothing to show.

Case 2: If $t_n \downarrow t$, then for large $n$ we have $e^-(t_n) = e^-(t)$, and hence
$w(t) = x(t) = \lim a^-(e^-(t_n)) = \lim a^-(e^-(t)) = a^-(e^-(t)) = x^-(t) = w^-(t)$.

Case 3: is empty since $d$ is strictly increasing.

Case 4:
Let $t_n \uparrow t$. For large $n$, $e(t_n) = e(t)$. Hence
$y^-(t) = \lim a(e(t_n)) = a(e(t)) = y(t)$.
Letting $t_n \downarrow t$, we see that $x(t) = \lim a^-(e(t_n)) = a(e(t)) = y(t)$, and so
$w(t) = x(t) = y(t) = y^-(t) = w^-(t)$.

Case 5:
Is empty since $d$ is strictly increasing and r.c.l.l.

Case 6: We have $w(t) = x(t)$ and $w^-(t) = x^-(t)$.
Since $e$ is (left- and right-) increasing at $t$, $d$ must be continuous at $e(t)$, i.e.\ $e(t) \notin \disc(d)$.  By assumption, $e(t) \notin \disc(a)$, and hence
$x^-(t) = a^-(e(t)) = a(e(t))$.
Letting $t_n \downarrow t$ and using that $e(t_n) \downarrow e(t)$, we see $x(t) = \lim a^-(e(t_n)) = a(e(t))$, i.e.\ $x^-(t) = x(t)$.
\end{proof}

In the above lemma, the assumption that $d$ be strictly increasing is unnecessary; it was only used for convenience.
We now prove the key ingredient in the continuous mapping argument.
Throughout, we use the terminology of Whitt \cite{Whitt2010}.

\begin{prop} \label{prop:mapping}
Endow the domain $\mathbb D(\Rd) \times \mathbb D_{u, \uparrow}$ of $\Phi$ with the trace of the $J_1$ topology on $\mathbb D(\R^{d+1})$.
Endow the codomain $\mathbb D(\Rd)$ of $\Phi$ with the $M_1$ topology.
Then $\Phi$ is continuous at every point $(a,d)$ which is such that $d$ is strictly increasing.
\end{prop}
\begin{proof}
Let $(a,d) \in \mathbb D(\Rd) \times \mathbb D_{u, \uparrow}$ be such that $d$ is strictly increasing, and consider a sequence $\{(a_n,d_n)\}_{n\in \mathbb N} \subset \mathbb D(\Rd) \times \mathbb D_{u, \uparrow}$ such that $\lim (a_n,d_n) =  (a,d)$ with respect to the $J_1$ topology.
Write $e_n$, $g_n$, $h_n$, $x_n$ and $y_n$ for the paths which are associated with $a_n$ and $d_n$ in the same fashion as $e, g, h, x$ and $y$ are associated with $a$ and $d$.

Recall the following equivalent description of convergence in $\mathbb D(E)$, where $E$ is any separable metric space (for us, $E = \Rd$ or $E = \R$)
\cite[Th~3.6.5]{EthierKurtz}:
With respect to the topology $J_1$, we have $\xi_n \to \xi$ if and only if the following three statements hold simultaneously for every $t \ge 0$:
\begin{enumerate}
 \item[(J1)]
 For every sequence $t_n \to t$ we have
 $\xi_n(t_n) \to \{\xi^-(t), \xi(t)\}$.
 \item[(J2)]
 Let $t_n \to t$ be such that $\xi_n(t_n) \to \xi^-(t)$.
 If $s_n \le t_n$, then $\xi_n(s_n) \to \xi^-(t)$.
 \item[(J3)]
 Let $t_n \to t$ be such that $\xi_n(t_n) \to \xi(t)$.
 If $u_n \ge t_n$, then $\xi_n(u_n) \to \xi(t)$.
\end{enumerate}
Note that the first statement means that $\xi_n(t_n)$ has at most two limit points, either $\xi^-(t)$ or $\xi(t)$.
If $t \notin \disc(\xi)$, then the three statements collapse to $\xi_n(t_n) \to \xi(t)$.
In \cite{StrakaHenry}, it was pointed out that (J1) may be replaced by
\begin{enumerate}
 \item[(J1')]
 For every sequence $t_n \to t$ we have
 $\xi^-_n(t_n) \to \{\xi^-(t), \xi(t)\}$,
\end{enumerate}
which we will use below.
We also give an equivalent description of convergence in $\mathbb D(E)$ with respect to the topology $M_1$, which can be derived from \cite[Th~12.5.1(v)]{Whitt2010}:
With respect to $M_1$, we have $\xi_n \to \xi$ if and only if the following two statements hold simultaneously:
\begin{enumerate}
 \item[(M1)]
 Let $t \notin \disc(\xi)$. If $t_n \to t$, then $\xi_n(t_n) \to \xi(t)$.
 \item[(M2)]
 Let $t \in \disc(\xi)$. If $s_n \to t$, $t_n \to t$, $u_n \to t$ where $s_n < t_n < u_n$, then $\|\xi_n(t_n) -  [\xi_n(s_n),\xi_n(u_n)] \| \to 0$.
\end{enumerate}
Note that in (M2), the brackets denote a line segment as defined earlier, and $\| \cdot \|$ denotes the Euclidean norm in $\Rd$.

The proof of Lemma \ref{lem:wcont} shows that $t \notin \disc(w)$ in the cases 1-5, as the assumption $\disc(a) \subset \disc(d)$ was only used in Case 6.
Hence we need to show (M1) in the cases 1-5 and (M1) \& (M2) in Case 6.
From previous work \cite[Prop~ 2.3]{StrakaHenry} we know that
\begin{align} \label{eq:previous}
x_n  \to x \text{ and } y_n  \to y \text{ in } \mathbb D(\Rd) \text{ and }
g_n  \to g \text{ and } h_n  \to h \text{ in } \mathbb D(\R)
\end{align}
with respect to $J_1$.

Case 1: $e, g,h,x$ and $y$ are all constant in a neighbourhood of $t$ and thus continuous.
Hence $g_n(t_n) \to g(t)$, $h_n(t_n) \to h(t)$, $x_n(t_n) \to x(t)$ and $y_n(t_n) \to y(t)$.
Since $g(t) < t < h(t)$, for large $n$ we have $g_n(t_n) < t_n < h_n(t_n)$, and hence $t_n \notin \mathcal R(d_n)$.
It follows that $w_n(t_n) \to w(t)$.

Case 2:
We have $g^-(t) = g(t) = t = h^-(t) <  h(t)$ and, as noted before, $w(t) = w^-(t) = x(t) = x^-(t) = y^-(t)$.
Hence $x_n(t_n) \to x(t)$, and $y_n(t_n) \to \{y^-(t), y(t)\}$ by (J1) applied to $x_n$ and $y_n$.
Split $t_n$ into two subsequences $t'_n$ and $t''_n$ such that
$y_n(t'_n) \to y^-(t)$ and $y_n(t''_n) \to y(t)$.
Equation \eqref{eq:interpol} immediately yields $w_n(t'_n) \to w(t)$.
It remains to show $w_n(t''_n) \to w(t)$.
(J1) applied to the sequence $(a_n,d_n)$ at $\tau_n = e_n(t''_n)$ then yields
$(a_n(\tau_n),d_n(\tau_n)) \to \{(a^-(\tau),d^-(\tau)),(a(\tau),d(\tau))\}$ where $\tau = e(t)$.
But we have $a_n(\tau_n) = a_n(e_n(t''_n)) = y_n(t''_n) \to y(t) = a(\tau)$, and hence it follows that $d_n(\tau_n) \to d(\tau)$, which means
$h_n(t_n) = d_n(\tau_n) \to d(\tau) = h(t)$.
By continuity at $t$, we have $g_n(t_n) \to g(t)$ and $x_n(t_n) \to x(t)$.
In the second case of $w_n(t''_n)$ in \eqref{eq:levywalkpath}, the fraction converges to $0$, and hence both cases then yield the limit $x(t) = w(t)$.

Case 3:
is empty since $d$ is strictly increasing and r.c.l.l.

Case 4:
We have $g^-(t) < g(t) = t = h^-(t) = h(t)$.
As noted before, we have $w(t) = w^-(t) = x(t) = y(t) = y^-(t)$.
Then $y_n(t_n) \to y(t)$, and $x_n(t_n) \to \{x^-(t), x(t)\}$.
Again split $t_n$ into two subsequences $t'_n$ and $t''_n$ such that
$x_n(t'_n) \to x^-(t)$ and $x_n(t''_n) \to x(t)$.
From \eqref{eq:interpol}, it follows that $w_n(t''_n) \to w(t)$.
It remains to show $w_n(t'_n) \to w(t)$.
First, we observe that it is possible to choose a sequence $\eps_n \downarrow 0$ so that
$(x_n(t'_n), g_n(t'_n)) = (a_n^-(e_n^-(t'_n+)), d_n^-(e_n^-(t'_n+)))$ and
$(x_n^-(t'_n+\eps_n), g_n^-(t'_n + \eps_n)) = (a_n^-(e_n^-(t'_n + \eps_n)), d_n^-(e_n^-(t'_n+\eps_n)))$ yield the same limit as $n \to \infty$.
Let $\tau_n = e_n^-(t'_n + \eps_n)$. Then $\tau_n \to \tau = e(t)$.
(J1') applied to the sequence $(a^-_n, d^-_n)$ at $\tau_n$ yields
$(a_n^-(\tau_n), d_n^-(\tau_n)) \to \{ (a^-(\tau), d^-(\tau)), (a(\tau),d(\tau)) \}$.
But $a_n^-(\tau_n) = x^-_n(t'_n) \to x^-(t) = a^-(\tau)$, and hence $d_n^-(\tau_n) \to d^-(\tau)$.
This means $g_n(t'_n) \to g^-(t)$.
If $g_n(t'_n) \in \mathcal R(d_n)$ for infinitely many $t'_n$, then
$g_n(t'_n) = t'_n$ for infinitely many $t'_n$, which contradicts
$g_n(t'_n) \to g^-(t) < t$.
Hence for large $n$ we have $g_n(t'_n) \notin \mathcal R(d_n)$.
The second case in the definition \eqref{eq:levywalkpath} of $w_n(t'_n)$ applies.
The fraction converges to $1$, and $w(t'_n) \to y(t) = w(t)$.

Case 5:
Is empty since $d$ is strictly increasing and r.c.l.l.

Case 6:
We have $g^-(t) = g(t) = t = h^-(t) = h(t)$, and moreover
$w(t) = x(t) = y(t) = a(e(t))$ and $w^-(t) = x^-(t) = y^-(t) = a^-(e(t))$.
Assuming that $t \notin \disc(w)$ yields $x^-(t) = x(t)$ and $y^-(t) = y(t)$,
which means $x_n(t_n) \to x(t)= w(t)$ and $y_n(t_n) \to y(t) = w(t)$.
Then $w_n(t_n) \to w(t)$ by \eqref{eq:interpol}.

Now we assume that $t \in \disc(w)$ and prove (M2).
We have $\tau := e(t) \in \disc(a)$.
By $J_1$-convergence of $a_n \to a$, there exists a sequence $\tau_n \to \tau$
such that $a_n(\tau_n) \to a(\tau)$ and $a_n^-(\tau_n) \to a^-(\tau)$.

Consider now the (possibly finite or empty) subsequence $t'_n$ of $t_n$ for which
$y_n(t'_n) \to y^-(t)$.
Then necessarily $e_n(t'_n) < \tau_n$ for all but finitely many $n$,
or else (J3) would contradict $y_n(t'_n) \to y^-(t)$.
Choose a sequence $\eps_n \downarrow 0$ such that $e_n(t'_n) + \eps_n < \tau_n$.
Choose another sequence $\eps'_n \downarrow 0$
such that $\lim x_n(t'_n) = \lim x_n^-(t'_n+\eps'_n)$.
Define $\eps''_n = \min(\eps'_n,\eps_n)$.
Then
$\lim x_n(t'_n) = \lim x_n^-(t'_n+\eps''_n)
= \lim a_n^-(e_n^-(t'_n + \eps''_n))$,
and since $e_n^-(t'_n + \eps''_n) \le e_n(t'_n + \eps''_n) < \tau_n$,
the last limit equals $a^-(\tau) = y^-(t)$.
Hence we have shown that $x_n(t'_n) \to y^-(t) = w^-(t)$.
Recalling that $w_n(t'_n) \in [x_n(t'_n),y_n(t'_n)]$, this means $w_n(t'_n) \to w^-(t)$.
Let $s'_n$ be the subsequence of $s_n$ which matches $t'_n$.
Since $s'_n < t'_n$, we have $e_n(s'_n) \le e_n(t'_n)$ and $e_n^-(s'_n) \le e_n(t'_n)$.
(J2) then implies $y_n(s_n) \to y^-(t)$ and $x_n(s_n) \to y^-(t)$.
Using \eqref{eq:interpol}, this means $w_n(s'_n) \to y^-(t) = w^-(t)$.
(M2) hence holds for $w_n$ and the subsequence $t'_n$, since
$w_n(t'_n)$ and $w_n(s'_n)$ have the same limit.

Now consider the subsequence $t''_n$ of $t_n$ defined by $x_n(t''_n) \to x(t) = y(t)$.
With an argument dual to the above, it can be shown that then $y_n(t''_n) \to y(t)$ as well.
By \eqref{eq:interpol}, $w_n(t''_n) \to y(t)$,
and using (J3), one shows $w_n(u''_n) \to y(t)$.
Since $w_n(t''_n)$ and $w_n(u''_n)$ have the same limit, (M2) follows for the subsequence $t''_n$.

It remains to prove (M2) for the subsequence of remaining elements of $t_n$, $s_n$ and $u_n$. For ease of notation we denote these elements again by $t_n$, $s_n$ and $u_n$.
We have $y_n(t_n) \to y(t) = w(t)$ and $x_n(t_n) \to x^-(t) = w^-(t)$.

Firstly, consider the members $t_n$, $s_n$ and $u_n$ for which
$g_n(t_n) \le s_n < t_n < u_n < h_n(t_n)$.
From the definition \eqref{eq:levywalkpath} of $w$,
we see that $w_n(t_n) \in [w_n(s_n), w_n(u_n)]$,
i.e.\ $\|w_n(t_n) - [w_n(s_n), w_n(u_n)]\| = 0$ and (M2) holds.

Secondly, for the members which satisfy
$s_n < g_n(t_n) \le t_n < u_n < h_n(t_n)$,
we have $w_n(t_n) \in [x_n(t_n), w_n(u_n)]$ by definition \eqref{eq:levywalkpath} of $w$.
Then (M2) will hold if $\lim w_n(s_n) = \lim x_n(t_n) = x^-(t)$.
Indeed, we have that $(s_n,t_n] \cap \mathcal R(d_n)$ is non-empty,
and so
$e_n(s_n) < e_n(t_n)$.
As noted before, there exist $\eps_n \downarrow 0$  such that
$a^-(e(t)) = x^-(t) = \lim x_n(t_n) = \lim a_n^-(e_n^-(t_n+))
= \lim a_n^-(e_n^-(t_n+\eps_n))$.
Now $e_n^-(t_n+\eps_n) \ge e_n(t_n) > e_n(s_n)$.
Hence by (J2) applied to $a_n\to a$ we have
$y_n(s_n) = a_n(e_n(s_n)) \to a^-(e(t)) = x^-(t)$.
Moreover, $x_n(s_n) \to x^-(t)$ by (J2) applied to $x_n\to x$.
By \eqref{eq:interpol}, $w_n(s_n) \to x^-(t)$.

Thirdly, for the members for which
$g_n(t_n) \le s_n < t_n \le h_n(t_n) \le u_n$,
we have $w_n(t_n) \in [w_n(s_n),y_n(t_n)]$ by definition \eqref{eq:levywalkpath} of $w$.
Then (M2) will hold if $\lim w_n(u_n) = \lim y_n(t_n) = y(t)$.
Indeed, proceed similarly to the previous case.
We have that $(t_n,u_n] \cap \mathcal R(d_n)$ is non-empty (note that $t_n = h_n(t_n)$ means that $t_n$ is a right-limit point of $\mathcal R(d_n)$), and so $e_n(t_n) < e_n(u_n)$.
We have $a(e(t)) = y(t) = \lim y_n(t_n) = \lim a_n(e_n(t_n))$.
Let $\eps_n$ be such that $\lim x_n^-(u_n + \eps_n) = \lim x_n(u_n)$.
Now $e_n^-(u_n+\eps_n) \ge e_n(u_n) > e_n(t_n)$.
Hence by (J3) applied to $a_n\to a$ we have
$\lim x_n(u_n) = \lim a_n^-(e_n^-(u_n+))
= \lim a_n^-(e_n^-(u_n+\eps_n)) = a(e(t)) = y(t)$.
Moreover, (J3) applied directly to $y_n\to y$ yields $y_n(u_n) \to y(t)$.
By \eqref{eq:interpol}, $w_n(u_n) \to y(t)$.

Fourthly, consider the members  for which
$s_n < g_n(t_n) \le t_n \le h_n(t_n) \le u_n$.
As in ``secondly'' and ``thirdly'', it follows that
$w_n(s_n) \to x^-(t)$ and $w_n(u_n) \to y(t)$.
Since $w_n(t_n) \in [x_n(t_n),y_n(t_n)]$ and $x_n(t_n) \to x^-(t)$ and $y_n(t_n) \to y(t)$, (M2) follows.
As the above four cases are exhaustive for the sequences $t_n$, $s_n$ and $u_n$, the proof is finished.
\end{proof}

\begin{lemma} \label{lem:measurablePhi}
 The mapping $\Phi$ is Borel measurable.
\end{lemma}
\begin{proof}
We first note that the composition mapping is measurable, as shown in \cite[p.232]{Billingsley1968}, and hence the two mappings
$D(\Rd) \times D_{u,\uparrow} \ni (a,d) \mapsto y = a \circ e \in D(\Rd)$
and
$D(\Rd) \times D_{u,\uparrow} \ni (a,d) \mapsto h = d \circ e \in D(\Rd)$
are measurable.
Next, we show that the mapping $D(\Rd) \times D_{u,\uparrow} \ni (a,d) \mapsto x = (a^- \circ e^-)^+ \in D(\Rd)$ is measurable.
Since the finite-dimensional sets in $D(\Rd)$ generate its Borel $\sigma$-field, it suffices to show that for every $t \ge 0$, the mapping $(a,d) \mapsto x(t)$ is measurable.
For a sequence $t_n \downarrow t$, we have $x(t) = \lim a^-(e^-(t_n))$, and hence it suffices to show that the mapping $(a,d) \mapsto a^-(e^-(t))$ is measurable for every $t \ge 0$.
For $k \in \mathbb N$, let $e^-_k(t)$ be the largest ratio $i/k$ not larger than $e^-(t)$; we then have $\lim_{k \to \infty} a^-(e^-_k(t)) = a^-(e^-(t))$, and so it suffices to show that for every $t \ge 0$ and $k \in \mathbb N$, the mapping $(a,d) \mapsto a^-(e^-_k(t))$ is measurable. Fix $k$ and $t$, and let $B$ be a Borel set in $\Rd$.
Then the set $\{(a,d): a^-(e^-_k(t)) \in B\}$ is the union of
$\{e^-_k(t) = 0\} \cap \{a(0) \in B\}$
with the sets
\begin{align*}
 \left\lbrace (a,d) : \frac{i}{k} \le e^-(t) < \frac{i+1}{k} \right\rbrace
 \cap \left\lbrace (a,d): a\left( \frac{i}{k} \right) \in B \right\rbrace,
\end{align*}
where $i$ is running through $\mathbb Z$. Since $d\mapsto e$ is measurable \cite[Th.13.6.1]{Whitt2010} and $e^-(t) = \lim e(t_n)$ where $t_n \uparrow t$, the mapping $d \mapsto e^-(t)$ is measurable, and hence the above sets lie in the Borel $\sigma$-field of $D(\Rd)$. This shows that $(a,d) \mapsto x$ is measurable, and similarly to the above one shows that the map
$D(\Rd) \times D_{u,\uparrow} \ni (a,d) \mapsto g = (d^- \circ e^-)^+$ is measurable.

Next, we note that due to the right-continuity of $d$, we have
\begin{align*}
 \{(a,d): t \in \mathcal R(d)\}
 = \bigcap_{n\in\mathbb N} \bigcup_{q \in \mathbb Q}\{(a,d): d(q) \in [t, t+n^{-1}]\}
\end{align*}
which is thus seen to be a Borel set in $D(\Rd)$.

Finally, on the set $\{t \notin \mathcal R(d)\}$, $w(t)$ is calculated from $x(t),y(t),g(t)$ and $h(t)$ via a sum, product and quotient; hence
the set $\{(a,d): w(t) \in B\}$ is seen to be a Borel set in $D(\Rd)$, since it
can be written as the union of
$\{t \in \mathcal R(d)\} \cap \{x(t) \in B\}$ with
$\{t \notin \mathcal R(d)\} \cap \{w(t) \in B\}$.
\end{proof}

We can now apply the continuous mapping theorem.

\begin{theorem} \label{th:limit}
Let $S(n)$ and $T(n)$ be the cumulative processes as in Section 2, and define the processes
\begin{align} \label{eq:AnDn}
 A_n(t) &= B(n) S(\lfloor nt \rfloor),
&D_n(t) &= b(n) T(\lfloor nt \rfloor)
\end{align}
where $B(n)$ and $b(n)$ are a spatial and a temporal scaling sequence.
As $n \to \infty$, suppose that $\{(A_n(t),D_n(t)\}_{t \ge 0}$ converges weakly in $\mathbb D(\spctim)$
with respect to the topology $J_1$ to the process $\{(A(t),D(t))\}_{t \ge 0}$.
Assume that $D(t)$ has strictly increasing sample paths a.s.
Then the rescaled L\'evy Walk $\{W_n(t)\}_{t \ge 0}$ given by
\begin{align*}
W_n(t) = B(n) W(t/b(n))
\end{align*}
converges weakly in $\mathbb D(\Rd)$ with respect to the $M_1$ topology to the limiting process $ \{ L(t) \}_{t \ge 0}$ given by
\begin{align} \label{eq:defL}
 L(t) &=
 \begin{cases}
  X(t)  &\text{ if } t \in \mathcal R(D) \\
  X(t) + \dfrac{t-G(t)}{H(t) - G(t)}(Y(t) - X(t))  &\text{ if } t\notin \mathcal R(D),
 \end{cases}
 \end{align}
 \begin{align*}
\mathcal R(D) &= \{D(t): t \ge 0\} \subset [0,\infty),
& E(t) &= \inf\{ r: D(r) > t\}, \\
 X(t) &= A^-(E(t+)),
&Y(t) &= A(E(t)), \\
 G(t) &= D^-(E(t+)),
&H(t) &= D(E(t)).
\end{align*}
\end{theorem}
Recall that the l.c.r.l.\ process $\{A^-(t)\}_{t \ge 0}$ is given by
$A^-(t) = A(t-)$ if $t > 0$ and $A^-(0) = A(0)$, and similarly for $D^-(t)$.

\begin{proof}
Our first step is to show $\Phi(A_n,D_n)(t) = W_n(t)$.
We define the processes
 \begin{align*}
E_n(t) &:= \inf\{r: D_n(r) > t\} = n^{-1}(N_{ t / b(n)} + 1)\\
G_n(t) &:= D_n^-(E_n^-(t+)) = b(n) T(N_{ t/b(n)}) \\
H_n(t) &:= D_n(E_n(t)) = b(n) T(N_{ t / b(n)} + 1) \\
X_n(t) &:= A_n^-(E_n^-(t+)) = B(n) S(N_{ t / b(n)}) \\
Y_n(t) &:= A_n(E_n(t)) = B(n) S(N_{t / b(n)} + 1).
\end{align*}
If $t \notin \mathcal R(D_n)$, then $G_n(t) < t < H_n(t)$, and
\begin{align*}
&\Phi(A_n,D_n)(t)
= X_n(t) + \frac{t-G_n(t)}{H_n(t) - G_n(t)} (Y_n(t) - X_n(t)) \\
&= B(n) S(N_{t / b(n)}) + \frac{t - b(n)T(N_{t / b(n)})}{b(n)J_{N(t/b(n))+1}}
B(n) J_{N(t/b(n))+1}\Lambda_{N(t/b(n))+1} \\
&= B(n) \left( S(N_{t/b(n)})
+ [t/b(n) - T(N_{t/b(n)})] \Lambda_{N(t/b(n))+1} \right) \\
&= W_n(t).
\end{align*}
If $t \in \mathcal R(D_n)$, then one finds
$t = b(n) T(k)$, for some $k \in \mathbb N_0$, and then
$t / b(n) = T(k)$ implies $k = N_{t / b(n)}$.
But then $t - T(N_{t/b(n)}) = 0$, and
\begin{align*}
W_n(t) = X_n(t) = \Phi(A_n,D_n)(t).
\end{align*}

Next, we check the assumptions of the Continuous Mapping Theorem,
\cite[p.30]{Billingsley1968}.
Endow $\mathbb D(\spctim)$ and $\mathbb D(\Rd)$ with their Borel $\sigma$-fields (it is irrelevant which of the topologies $J_1$ or $M_1$, see \cite{Whitt2010}).
Let $\mathbf P_n$ and $\mathbf P$ denote the laws on $\mathbb D(\spctim)$
of the processes associated with \eqref{eq:AnDn} and $\{(A_t,D_t)\}_{t \ge 0}$.
Then our assumption reads as the weak convergence
$\mathbf P_n \Rightarrow \mathbf P$, as $n \to \infty$.
The domain of the path map $\Phi$ defined in \eqref{eq:levywalkpath}
is a Borel set of $\mathbb D(\spctim)$ \cite[Lem~2.1]{StrakaHenry}.
It is then straightforward to recast the topology, $\sigma$-field and probability measures $\mathbf P_n$ and $\mathbf P$ to the domain of $\Phi$.
We redefine these new probability measures again as $\mathbf P_n$ and $\mathbf P$, for ease of notation.
By \cite[Cor~3.2]{EthierKurtz}, $\mathbf P_n \Rightarrow \mathbf P$.
The mapping $\Phi$ is measurable by Lemma~\ref{lem:measurablePhi}.
The sample paths $D(t)$ are strictly increasing $\mathbf P$ a.s.
Hence by Proposition~\ref{prop:mapping} the set of discontinuities of $\Phi$ is a $\mathbf P$-null set.

The statement of the Continuous Mapping Theorem is then that
$\mathbf P_n \circ \Phi^{-1} \Rightarrow \mathbf P \circ \Phi^{-1}$, which is weak convergence in the codomain $\mathbb D(\Rd)$ with respect to the topology $M_1$.
We have checked earlier in this proof that $\mathbf P_n \Phi^{-1}$ is the law of the process $\{W_n(t)\}_{t \ge 0}$, and it is straightforward to see that $\mathbf P \Phi^{-1}$ is the law of the process $\{L(t)\}_{t \ge 0}$.
\end{proof}

\begin{cor}
\label{cor1}
Let $\beta \in (0,1)$.
Then as $c \to \infty$, the rescaled L\'evy Walk
\begin{align} \label{eq:rescaledLW01}
 \left\lbrace \frac{1}{c} W(ct) \right\rbrace_{t \ge 0}
\end{align}
converges  in $\mathbb D(\Rd)$ with respect to the topology $U$ of uniform convergence to the process $\{L(t)\}_{t \ge 0}$ given by \eqref{eq:defL} and the L\'evy process $\{(A(t),D(t))\}_{t \ge 0}$ with symbol $\psi(k,s)$ as in \eqref{eq5}.
\end{cor}
\begin{proof}
In view of \eqref{eqP1} we have 
\begin{align*}
 \left\lbrace b(c) \left( S(\lfloor ct \rfloor), T(\lfloor ct \rfloor) \right) \right\rbrace_{t \ge 0}
 \stackrel{J_1}{\longrightarrow}
 \{(A(t),D(t))\}_{t \ge 0}
\end{align*}
as $c\to\infty$. 
Then by Theorem \ref{th:limit}, the process $\{W_c(t)\}_{t \ge 0}$ given by
\begin{align}\label{eqCC1}
 W_c(t) = b(c) W(b(c)^{-1}t)
\end{align}
converges to $\{L(t)\}_{t \ge 0}$ with respect to the topology $M_1$.
But the sample paths of the limiting process $\{L(t)\}_{t \ge 0}$ are a.s.\ continuous by Lemma~\ref{lem:wcont}, and hence convergence also holds in the stronger topology $U$.
Since $b(c)^{-1}$ is regularly varying with index $1/\beta$ there exists a function $\tilde b(c)$ regularly varying with index $\beta$ such that $b(\tilde b(c))^{-1}\sim c$ as $c\to\infty$.  Now replace $c$ by $\tilde b(c)$ in \eqref{eqCC1} and the result follows.
\end{proof}
\begin{remark}
The scaling limit $L(t)$ in the above corollary is 1-selfsimilar.
Thus, its variance in the one-dimensional symmetric case equals
$\text{Var}(L(t))=t^2 E[L^2(1)]$. We have analogous asymptotic behavior of the variance for the corresponding L\'evy walk process $\text{Var}(W(t))\propto t^2$ as
$t\rightarrow\infty$, see \cite{Klafter2}. Such quadratic in time scaling of the variance is typical for
\emph{ballistic motion}. 
This kind of motion is characteristic for a Brownian particle at the early stage of its movement right after a collision, 
due to the inertia of the particle \cite{Huang}.
After this initial stage, a Brownian particle makes the transition from ballistic to diffusive (linear in time growth of variance) regime.
\end{remark}

\begin{remark}
Note that the trajectories of the limit process $L(t)$ in the above corollary are \emph{continuous}. Thus,
it is a different process than the one obtained as a limit of the corresponding CTRW (see Example 5.4. in \cite{coupleCTRW}).
The difference is even more evident in the particular one-dimensional case with $X_i = J_i$ (jumps equal to waiting times).
Then, the L\'evy walk is just a deterministic linear function $W(t)=t$. So is its scaling limit $L(t)=t$.
However, the scaling limit of the corresponding CTRW is a jump process
with one-dimensional distribution given by the generalized arcsine (Beta) law \cite{coupleCTRW}.
\end{remark}

\begin{cor}
Let $\beta \in (1,2)$, and assume that the distribution of $J$ belongs to the domain of attraction of a $\beta$-stable law. Let $\mu = \mathbf E[J]$. Assume further that $E[\Lambda_i]=0$. Then there exists a scaling function $b(c)$ such that, as $c \to \infty$, the rescaled L\'evy Walk
\begin{align} \label{eq:rescaledLW12}
 \left\lbrace b(c) W(\mu ct) \right\rbrace_{t \ge 0}
\end{align}
converges in $\mathbb D(\Rd)$ with respect to the $M_1$ topology to the
$\beta$-stable process $\{A(t)\}_{t \ge 0}$ with spectral measure $\lambda$.
\end{cor}

\begin{proof} 
By assumption on $J$ there exists a regularly varying function $b(c)$ with index $-1/\beta$ such that 
$b(c)T(\lfloor c\rfloor)-cb(c)\mu\Longrightarrow D$ as $c\to\infty$ for some $\beta$-stable random variable $D$. 
Moreover, for any $u>0$ and Borel sets $V\subset\Sd$ we have 
\begin{multline*}
cP\{b(c)J\Lambda \in B(u,V)\} = cP\{b(c)J>u, \Lambda\in V\}\\ =cP\{b(c)J>u\}\lambda(V)\to \phi_D(u,\infty)\lambda(V)
\end{multline*}
 where $\phi_D$ is the L\'evy measure of $D$. Since $E[J \Lambda]=E[J]\cdot E[\Lambda]=0$ this implies that
 $b(c)S(\lfloor c\rfloor)\Longrightarrow A$, where $A$ is multivariate $\beta$-stable with spectral measure $\lambda$.
 Using Theorem 16.14 in \cite{kallenberg} this implies 
 \[ \bigl\{b(c)S(\lfloor ct\rfloor)\bigr\}_{t\geq 0}\overset{J_1}{\longrightarrow} \bigl\{A(t)\bigr\}_{t\geq 0}. \]
 Moreover, by the strong law of large numbers we have $c^{-1}T(\lfloor c\rfloor)\to\mu$ almost surely and hence, using Theorem 16.14 in \cite{kallenberg} again we have 
 \[ \frac 1{\mu c}T(\lfloor ct\rfloor)\overset{J_1}{\longrightarrow} \{ t\}_{t\geq 0}. \]
 Since the latter limit process is deterministic we get
\begin{align*}
 \left\lbrace \left( b(c) S(\lfloor ct \rfloor), (\mu c)^{-1} T(\lfloor ct \rfloor) \right) \right\rbrace_{t \ge 0} \stackrel{J_1}{\longrightarrow} \left\lbrace (A(t), t) \right\rbrace_{t \ge 0}.
\end{align*}
Theorem~\ref{th:limit} then yields the statement.
\end{proof}

\begin{remark}[The superdiffusive regime]
For $\beta \in (1,2)$, the variance of the L\'evy walk grows as $\text{Var}(W(t))\propto t^{3-\beta}$ as $t\rightarrow\infty$, see \cite{Klafter2}. 
This is the so-called \emph{superdiffusive regime}. However, the second moment of the limit process $A(t)$
is infinite. Thus, the situation is different from the previously analyzed case $\beta \in (0,1)$, where the ballistic regime was observed both for L\'evy walk and its scaling limit.
\end{remark}

\begin{remark}[Pathwise variation]
From the definition of the L\'evy Walk, it can be seen that the variation of almost every path $[0,t] \to W_t$
equals $t$. 
For the rescaled L\'evy Walk \eqref{eq:rescaledLW12} in the case 
$\beta \in (1,2)$, 
the variation at scale $c$ equals $b(c)\,\mu c t$. 
As the scaling function $b(c)$ is of regular variation with exponent $-1/\beta$, this is seen to grow indefinitely as $c \to \infty$, with approximate speed
$c^{1-1/\beta}$ as $c \to \infty$. This rings well with the limit process being a stable process, which is known to have sample paths with infinite variation almost surely.

In the case $\beta \in (0,1)$, the rescaled L\'evy Walk 
\eqref{eq:rescaledLW01} almost surely has sample paths with variation $t$ at all scales. 
We hence conjecture that the limiting process $L(t)$ also has sample paths with variation $t$ on the interval $[0,t]$.
\end{remark}

\section{The governing equation}

This section develops the governing equation of the one-dimensional marginals for the L\'evy walk scaling limit process $\{L(t)\}_{t\geq 0}$ obtained in Corollary \ref{cor1} as well as its solution. As in \cite{M3HPStri} and \cite{Jurlewicz}, we consider pseudo-differential equations in the variables $x$ and $t$, and solutions are defined in the ``mild'' sense. 
Our first step is to express the law of the L\'evy Walk Limit in terms of its ``model parameters'' $\beta \in (0,1)$ and $\lambda$.
For abbreviation, we introduce the $0$-potential density 
$$u(x,t) = \int_0^\infty p_u(x,t)\,du,$$
where $p_u(x,t)$ is the probability density of $(A_u,D_u)$, see Proposition \ref{psupp}.
The $0$-potential 
$U(dx,dt) = u(x,t)\,dx\,dt$
may be interpreted as the measure which assigns to a Borel set $C \subset \spctim$ the expected amount of time the process $(A_u,D_u)$ stays in $C$.
We note that $D_u$ is a subordinator, and hence $(A_u,D_u)$ is a transient Markov process, guaranteeing that the measure $U(dx, dt)$ is $\sigma$-finite \cite{chung2005markov}. 
Its Fourier-Laplace transform is given by $1/\psi(k,s)$.

\begin{theorem}
The law of the L\'evy Walk Limit process \eqref{eq:defL} with 
tail parameter $\beta \in (0,1)$ and spectral measure $\lambda$ is Lebesgue absolutely continuous with density 
\begin{align} \label{eq:levy-law}
\rho_t(x) = \int\limits_{S^{m-1}} \int\limits_0^t 
u(x-(t-g)\theta,g) \frac{(t-g)^{-\beta}}{\Gamma(1-\beta)}
\,dg\, \Lambda(d\theta).
\end{align}
\end{theorem}

\begin{proof}
First, we note that for every $t > 0$ one has $\pr(t \in \mathcal R(D)) = 0$, since the subordinator $D$ is strictly increasing with zero drift \cite{Bertoin04}.
The joint law of $X_t, G_t, Y_t$ and $H_t$ can be easily read off from \cite[Theorem 4.9]{StrakaHenry} (there, the joint law of $X_t$, $t-G_t$, $Y_t$ and $H_t - t$ is given): 
\begin{align}
\E[f(X_t,G_t,Y_t,H_t)]
= \int\limits_\rd \int\limits_0^t u(x,g)
\int\limits_\rd \int\limits_{t-g}^\infty 
f(x,g,x+z,g+w)\,\phi(dz,dw)\,dg\,dx,
\end{align}
which in somewhat more intuitive notation reads
\begin{align}
\pr\left( (X_t, G_t, Y_t, H_t) \in (dx, dg, dy, dh) \right)
= \boldsymbol 1(g < t < h) u(x,g) \phi (dy-x, dh-g) \,dx\,dg
\end{align}
We can hence use \eqref{eq:defL} to calculate 
\begin{align*}
&\E[f(L_t)] 
= \E\left[f\left (X_t + \frac{t-G_t}{H_t - G_t} (Y_t - X_t)\right )\right] \\
&= \int\limits_\rd \int\limits_0^t u(x,g) 
\int\limits_\rd \int\limits_{t-g}^\infty
f\left (x + \frac{t-g}{w}z\right )\,\phi(dz,dw)\,dg\,dx \\
&= \int\limits_\rd \int\limits_0^t u(x,g) 
\int\limits_0^\infty \frac{\beta r^{-1-\beta}}{\Gamma(1-\beta)}
\int\limits_{S^{m-1}} 
\boldsymbol 1(t < g + r)
f\left (x+\frac{t-g}{r} r\theta\right )
\,\Lambda(d\theta)\,dr\,dg\,dx
\end{align*}
from which  \eqref{eq:levy-law} follows.
\end{proof}

It is now easy to calculate the Fourier-Laplace transform of the law of $L_t$, at first with $\theta$ held fixed (or rather $\lambda$ concentrated at $\theta$):
\begin{align*}
&\int_0^\infty  \E[\exp(i\langle k,L_t \rangle)] e^{-st} \,dt
\\
&=\int_0^\infty \int_\rd \int_0^t u(x,g) \frac{(t-g)^{-\beta}}{\Gamma(1-\beta)} e^{-st} \exp(i\langle k, x+(t-g)\theta \rangle) \,dg\,dx\,dt
\\ 
&= \int_\rd \int_0^\infty u(x,g) \int_g^\infty \frac{(t-g)^{-\beta}}{\Gamma(1-\beta)} e^{-st} \exp(i\langle k, x+(t-g)\theta \rangle)
\,dt\,dg\,dx
\\ 
&= \int_\rd \int_0^t u(x,g) \int_0^\infty 
\frac{v^{-\beta}}{\Gamma(1-\beta)} e^{-s(g+v)} \exp(i\langle k, x+v\theta \rangle)
\,dv\,dg\,dx
\\ 
&= \int_\rd \int_0^t u(x,g) \exp(i\langle k,x\rangle - sg)
	\int_0^\infty \exp(-(s-i\langle k, \theta \rangle)v) \frac{v^{-\beta}}{\Gamma(1-\beta)}
\,dv\,dg\,dx
\\ 
&= \frac{1}{\psi(k,s)} (s-i\langle k, \theta \rangle)^{\beta-1}
\end{align*}
For general $\lambda$, we then have 
\begin{align} \label{eq:levy-FLT}
\int_0^\infty dt\, e^{-st} \E[\exp(i\langle k,L_t \rangle)]
= \frac{\int_{S^{m-1}}  (s-i\langle k,\theta\rangle)^{\beta-1}\lambda(d\theta)}{\psi(k,s)}
\end{align}

Now recall from \cite{M3HPStri} and \cite{Jurlewicz} that, given a weakly measurable family $h(dx,t)\ (t\geq 0)$ of bounded measures on $\rd$, we say that a family $\rho_t(dx)$ of probability measures is the {\it mild solution} to the pseudo-differential equation $\psi(i\nabla_x,\partial_t)\rho_t(dx)=h(dx,t)$ if its Fourier-Laplace transform (FLT) solves the corresponding algebraic equation. 
In this sense, we can prove:
\begin{theorem}
The density $\rho_t(x)$ of the L\'evy Walk limit process satisfies the following pseudo-differential equation on $\spctim$:
\begin{align}
\label{eq52}
\int_{\Sd}\bigl(\partial_t+\ip\theta{\nabla_x}\bigr)^\beta\rho_t(x)\ \lambda(d\theta)
=\frac{t^{-\beta}}{\Gamma(1-\beta)} \delta(\|x\|-t)\lambda\left(d\frac x{\|x\|}\right)
\end{align}
\end{theorem}

\begin{proof}
First, we note that the tempered distribution on the right-hand side can be represented as the Borel measure on $\spctim$
\begin{align}
(dx, dt) \mapsto \frac{t^{-\beta}}{\Gamma(1-\beta)}\int_{S^{m-1}} \delta(dx-t\theta)
\lambda(d\theta) \, dt.
\end{align}
By a straightforward calculation, this has the Fourier-Laplace transform
\begin{align}
\int_{S^{m-1}}  (s-i\langle k,\theta\rangle)^{\beta-1} \lambda(d\theta).
\end{align}
The statement then follows from \eqref{eq:levy-FLT} and \eqref{eqps1}.
\end{proof}

We close with an example of the developed theory:

\begin{example}
Let us consider the one-dimensional symmetric case, so that $\lambda(d\theta) = (\delta(d\theta - 1) + \delta(d\theta + 1))/2$.
It follows from \eqref{eq:levy-FLT} and \eqref{eq5} that 
\[
\bar \rho_s(k)=
\frac{1}{2}
\frac{(s-ik)^{\beta-1}+(s+ik)^{\beta-1}}{(s-ik)^{\beta}+(s+ik)^{\beta}}.
\]
Therefore
\[
\frac{\partial^2 \bar{\rho}_s(k)}{\partial k^2} \bigg | _{k=0}=-\frac{1-\beta}{s^3},
\]
and an inverse Laplace transform immediately recovers the ballistic scaling regime
\[
\E[L^2(t)]=t^2(1-\beta)/2.
\]
The pseudo-differential equation for $\rho_t(dx)$ then reads
\[
\frac{1}{2}\left [ \left( \frac{\partial}{\partial t} - \frac{\partial}{\partial x} \right)^\beta +
\left( \frac{\partial}{\partial t} + \frac{\partial}{\partial x} \right)^\beta \right] q(x,t) =
[\delta(x-t)+\delta(x+t)]\frac{t^{-\beta}}{\Gamma(1-\beta)}.
\]
Here, the operators $\left( \frac{\partial}{\partial t} \mp \frac{\partial}{\partial x} \right)^\beta$
are also called fractional material derivatives with Fourier-Laplace symbols $(s\mp ik)^\beta$. 
These were introduced in \cite{Sokolov_Metzler} as a fractional extension
of the standard material derivative.

\end{example}


\end{document}